\documentclass{article}
\usepackage[letterpaper,top=2cm,bottom=2cm,left=3cm,right=3cm,marginparwidth=1.75cm]{geometry}
\usepackage[normalem]{ulem}
\usepackage{amsmath}
\usepackage{amssymb}
\usepackage{amsthm}
\usepackage{amsfonts}
\usepackage{mathrsfs}
\usepackage{mathtools}
\usepackage{graphicx}
\usepackage{caption}
\usepackage{setspace}
\usepackage{amsthm}
\usepackage[colorlinks=true, allcolors=blue]{hyperref}
\usepackage{fullpage}
\usepackage{comment}
\usepackage{xcolor}
\usepackage{enumerate}
\usepackage{tikz}
\usepackage{float}
\usepackage[numbers]{natbib}
\usepackage[capitalize]{cleveref}

\usetikzlibrary{positioning}
\usetikzlibrary{fit,calc}

\newcommand{\N}{\mathbb N}

\newcommand{\gr}{\operatorname{gr}}
\newcommand{\AC}{AC}
\newcommand{\dvert}{d_{\text{v}}}
\newcommand{\dhoriz}{d_{\text{h}}}

\newcommand{\nhoriz}{N_{\text{h}}}

\newcommand{\eps}{\varepsilon}
\newcommand{\pgr}{\mathbf{PGR}}

\newtheorem{theorem}{Theorem}[section]
\newtheorem{lemma}[theorem]{Lemma}
\newtheorem{corollary}[theorem]{Corollary}
\newtheorem{prop}[theorem]{Proposition}

\newtheorem*{claim*}{Claim} %unnumbered claim 

\newtheorem{conj}[theorem]{Conjecture}

\theoremstyle{definition}
\newtheorem{definition}[theorem]{Definition}
\newtheorem{remark}[theorem]{Remark}
\newtheorem{question}[theorem]{Question}

\newcommand{\ghaura}[1]{\textcolor{blue}{\textbf{[Ghaura: #1]}}}
\newcommand{\krishna}[1]{\textcolor{brown}{\textbf{[Krishna: #1]}}}
\newcommand{\josh}[1]{\textcolor{violet}{\textbf{[Josh: #1]}}}
\newcommand{\jasper}[1]{\textcolor{teal}{\textbf{[Jasper: #1]}}} 
\newcommand{\lp}[1]{\textcolor{purple}{\textbf{[Logan: #1]}}} 
\newcommand{\ra}[1]{\textcolor{orange}{\textbf{[RA:} #1\textbf{]}}}

\newcommand{\lps}[2]{%
 \renewcommand{\ULthickness}{0.8pt}
 \textcolor{purple}{\sout{\textcolor{black}{#1}}}\textcolor{purple}{#2}
}

\begin{document}
\title{Ramsey numbers of grid graphs}
\author{Xiaoyu He\thanks{School of Mathematics, Georgia Institute of Technology, Atlanta, GA 30332. Email: xhe399@gatech.edu.} \and Ghaura Mahabaduge\thanks{Department of Mathematics, Massachusetts Institute of Technology, Cambridge, MA 02139. Email: ghaura\_m@mit.edu.}  \and Krishna Pothapragada\thanks{Department of Mathematics, Massachusetts Institute of Technology, Cambridge, MA 02139. Email: potha601@mit.edu.} \and Josh Rooney\thanks{Department of Mathematics, Harvard University, Cambridge, MA 02138. Email: joshuarooney@college.harvard.edu} \and Jasper Seabold\thanks{School of Mathematics, Georgia Institute of Technology, Atlanta, GA 30332. jasper.seabold@gatech.edu}}
\maketitle

\begin{abstract} 
Let the grid graph $G_{M\times N}$ denote the Cartesian product $K_M \square K_N$. For a fixed subgraph $H$ of a grid, we study the off-diagonal Ramsey number $\gr(H, K_k)$, which is the smallest $N$ such that any red/blue edge coloring of $G_{N\times N}$ contains either a red copy of $H$ (a copy must preserve each edge's horizontal/vertical orientation), or a blue copy of $K_k$ contained inside a single row or column. Conlon, Fox, Mubayi, Suk, Verstra\"ete, and the first author recently showed that such grid Ramsey numbers are closely related to off-diagonal Ramsey numbers of bipartite $3$-uniform hypergraphs, and proved that $2^{\Omega(\log ^2 k)} \le \gr(G_{2\times 2}, K_k) \le 2^{O(k^{2/3}\log k)}$. We prove that the square $G_{2\times 2}$ is exceptional in this regard, by showing that $\gr(C,K_k) = k^{O_C(1)}$ for any cycle $C \ne G_{2\times 2}$. We also obtain that a larger class of grid subgraphs $H$ obtained via a recursive blowup procedure satisfies $\gr(H,K_k) = k^{O_H(1)}$. Finally, we show that conditional on the multicolor Erd\H{o}s-Hajnal conjecture, $\gr(H,K_k) = k^{O_H(1)}$ for any $H$ with two rows that does not contain $G_{2\times 2}$.
\end{abstract}

\section{Introduction}

\
%Ramsey theory is concerned with the emergence of order within large systems. The classical theorem of Ramsey \cite{Ram30} asserts that for any positive integers \( r \) and \( s \), there exists a minimum number \( R(r, s) \) such that any graph on \( R(r, s) \) vertices contains a clique \( K_r \) or an independent set of size $s$. A similar result holds for hypergraphs. In particular, a $t$-uniform hypergraph (henceforth, $t$-graph) $G = (V,E)$ on $n$ vertices consists of a vertex set $V = [n]$ and an edge set $E \subseteq \binom{[n]}{t}$. Further, $K_n^{(t)}$ refers to the complete $t$-graph with $V = [n]$ and $E = \binom{[n]}{t}$. 

If $G$ and $H$ are $t$-uniform hypergraphs (henceforth $t$-graphs), then the Ramsey number $r(G, H)$ the minimum $N$ such any $t$-graph on $N$ vertices either contains $G$ or contains $H$ in its complement. Hypergraph Ramsey theory is focused on estimating the asymptotic growth of such functions, especially when one or both of $G, H$ is a complete $t$-graph, denoted $K_n^{(t)}$.

Perhaps the central problem of hypergraph Ramsey theory is to determine the tower height of the diagonal Ramsey number $r\left(K_n^{(3)}, K_n^{(3)} \right)$, for which the best known bounds are
\[
2^{\Omega(n^2)} \le r\left(K_n^{(3)}, K_n^{(3)} \right) \le 2^{2^{O(n)}},
\]
due to Erd\H{o}s, Hajnal, and Rado \cite{rkn3kn3}.
%The overarching problems in hypergraph Ramsey theory are the asymptotic growth of the ``diagonal" Ramsey number $r\left(K_n^{(t)}, K_n^{(t)} \right)$ as $n \to \infty$ and the ``off-diagonal" Ramsey number $r\left( K_m^{(t)},K_n^{(t)} \right)$ for $m > t$ fixed and $n \to \infty$. Through the work of Erd\H{o}s-Rado in \cite{erdosRado1952} and the stepping-up lemma of Erd\H{o}s-Hajnal, the diagonal case was reduced to $t = 3$. Further, for the off-diagonal case, the growth rate is known up to tower height except for $r\left( K_{t+1}^{(t)},K_n^{(t)} \right)$. However, by the efforts of \cite{MUBAYI2017168, MubayiSuk2018}, the tower height of the growth rate of $r\left( K_{t+1}^{(t)},K_n^{(t)} \right)$ would follow from showing the diagonal Ramsey number $r\left( K_n^{(3)},K_n^{(3)} \right)$ grows double exponentially in a power of $n$. Thus, attention has been turned towards the $t=3$ case. To improve the diagonal and off-diagonal Ramsey number bounds for the $3$-graph case, 
As a step towards this problem, Fox and the first author~\cite{Fox_2021} initiated the study of the Ramsey number $R( K_n^{(3)},S_k^{(3)} )$, where the $3$-uniform star $S_k^{(3)}$ is the $3$-graph with $k+1$ vertices, one of which is a distinguished ``center" vertex, and the $\binom{k}{2}$ hyperedges containing the center vertex. Roughly speaking, the star acts as a $3$-uniform model for the $2$-uniform complete graph, so this $R(K_n^{(3)},S_k^{(3)} )$ should be understood as a hybrid between the diagonal Ramsey numbers in uniformities $2$ and $3$. They settled the cases when $k$ is fixed in size and $n\rightarrow \infty$, and when both $k,n\rightarrow \infty$ together. Then, Conlon, Fox, Mubayi, Suk, Verstra\"ete and the first author \cite{conlon2023hypergraph} settled the final case when $n$ is fixed and $k\rightarrow \infty$ by reducing to Ramsey-type problems about grid graphs, which we believe to be of independent interest. %Their study of $r\left( K_4^{(3)},S_k^{(3)} \right)$ reduced the problem to so-called \emph{grid graphs}; a grid graph on $c$ columns and $r$ rows, denoted $G_{c \times r}$, is the Cartesian product $K_c \square K_r$.

%This has opened Ramsey theory on grid graphs as an area of interest. \ra{I think this transition sentence could be stronger. Something like ``Though the work of Conlon et al. on grid graphs was primarily motivated by questions about hypergraphs, it also opened up questions about grid Ramsey numbers as an interesting and independent research area."} 

We define a grid graph on $c$ columns and $r$ rows, denoted $G_{c \times r}$, to be the Cartesian product $K_c \square K_r$, and a \emph{spanning grid subgraph} on $c$ columns and $r$ rows to be a spanning subgraph of $G_{c \times r}$. Equivalently, a spanning grid subgraph is a graph on the vertices $\{(x, y)|x\in [c],y\in [r]\}$ where all edges are between vertices in the same row or the same column. That is, all edges are in the set $E^* = \{\{(x, y), (x', y)\}|x, x'\in [c],y\in [r]\} \cup \{\{(x, y),(x, y')\}|x\in [c],y, y'\in [r]\}$. The complement of a spanning grid subgraph $H \subseteq G_{c \times r}$, denoted $\overline H$,
is taken with respect to the grid graph $G_{c \times r}$ (i.e. $E(\overline H) = E^* \setminus E(H)$). A \emph{grid graph} on $c$ columns and $r$ rows is a graph on a subset of the vertices $[c] \times [r]$ where all edges belong to $E^*$ as above. Thus, a grid subgraph is any (not necessarily spanning) subgraph of a grid graph, together with the data of which sets of vertices lie in the same row or column.

For any grid subgraph $H \subseteq G_{c \times r}$, the \emph{(off-diagonal) grid Ramsey number} $\gr(H,K_k)$ is the least $N$ such that for any spanning grid subgraph $G$ of $G_{N\times N}$, either $G$ contains $H$ or its complement contains a clique of size $k$. Note that any complete graph in $G_{N\times N}$ must lie in a single row or column, so the goal is to avoid horizontal and vertical copies of $K_k$ in $\overline H$.
%\lp{Note that $K_k$ can be row or column, whereas $H$ must be in the right orientation, so there is a definitional asymmetry. Perhaps $\gr(H,k)$ is more honest.}\\
%\lp{In this paper, the vertex labeling always \textit{does} line up; you can remark on the alternative but I think for clarity, stick with the relevant version. (also, your embedding definition always make things line up).}

%\xh{--- Clean up above ---}

We study the asymptotics of $\gr(H, K_k)$ where $H$ is a fixed grid subgraph and $k \rightarrow \infty$.
In \cite{conlon2023hypergraph}, the case $H = G_{2\times 2}$ (henceforth called the square) was studied to bound $r( K_4^{(3)}, S_k^{(3)})$. They found that 
\[
2^{\Omega((\log k)^2)}\le \gr(G_{2 \times 2}, K_k) \le 2^{O(k^{2/3}\log k)}, 
\]
suggesting that $\gr(H, K_k)$ may have exotic superpolynomial growth rates in general. In this paper, we show that this example is unusual and that $R(H, K_k)$ grows polynomially in $k$ for a large family of $H$ that are square-free. In fact, it is possible that the square $G_{2\times 2}$ is the unique exception in this problem. We begin with a concrete statement about cycles.

\begin{definition}[Simple grid subgraph]
  A grid subgraph $H$ on $c$ columns and $r$ rows is \textit{simple} if its intersection with any row or column is connected. %on the vertices $(x_1,y_1),(x_2,y_2),\dots,(x_\ell,y_\ell)$, such that for $i < j$,
  %\begin{itemize}
  %  \item if $x_i = x_j$, then $x_k = x_i$ for every $i \leq k \leq j$; and
  %  \item if $y_i = y_j$, then $y_k = y_i$ for every $i \leq k \leq j$.
  %\end{itemize}
\end{definition}

Thus, a \textit{simple cycle} is a cycle whose embedding into the grid is simple, and a \textit{simple tree} is a tree whose embedding into the grid is simple. A typical example of a simple grid subgraph is the alternating cycle $AC_{2t}$, whose edges alternate in orientation, see \cref{fig:AC_8} below.

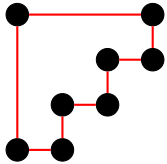
\begin{figure}[h!]
 \centering
  \begin{tikzpicture}[scale=0.6, mystyle/.style={circle, draw, fill=black, inner sep=3pt}]
  % Nodes
  \node[mystyle] (A) at (1,1) {};
  \node[mystyle] (B) at (2,1) {};
  \node[mystyle] (C) at (2,2) {};
  \node[mystyle] (D) at (3,2) {};
  \node[mystyle] (E) at (3,3) {};
  \node[mystyle] (F) at (4,3) {};
  \node[mystyle] (G) at (4,4) {};
  \node[mystyle] (H) at (1,4) {};

  % Edges – modify or add as needed
  \draw[red,thick] (A) -- (B);
  \draw[red,thick] (C) -- (B);
  \draw[red,thick] (C) -- (D);
  \draw[red,thick] (E) -- (D);
  \draw[red,thick] (E) -- (F);
  \draw[red,thick] (G) -- (F);
  \draw[red,thick] (G) -- (H);
  \draw[red,thick] (A) -- (H);

  \end{tikzpicture}
  \caption{The alternating cycle $AC_{8}$ of length $8$.}
  \label{fig:AC_8}
 \end{figure}

We show that all simple cycles except the square exhibit polynomial growth rate for $\gr(H, K_k)$. It will be convenient to define $\pgr$ to be the class of grid subgraphs $H$ satisfying $\gr(H, K_k) = k^{O_H(1)}$.

\begin{theorem}\label{thm: Simple Cycles are polynomial}
  For any simple cycle $H$ of length at least 5, $H \in \pgr$.
\end{theorem}

As in the previous works, our results have direct implications for Ramsey numbers of $3$-uniform hypergraphs. Let $C_t^{(3)}$ denote the $3$-uniform tight cycle on $t$ vertices (see \cref{fig:tight-cycle-8}). Conlon, Fox, Gunby, Mubayi, Suk, Verstr\"aete, and the first author and showed in \cite{conlon2025off} that the off-diagonal Ramsey number of tight-cycles versus the complete tripartite 3-graph exhibits super polynomial growth: $R(C_3^{(t)}, K_{n, n, n}^{(3)}) \geq 2^{\Omega_H(n \log n)}$ and this is tight. We show that replacing the complete tripartite 3-graph with the star instead gives a polynomial upper bound.

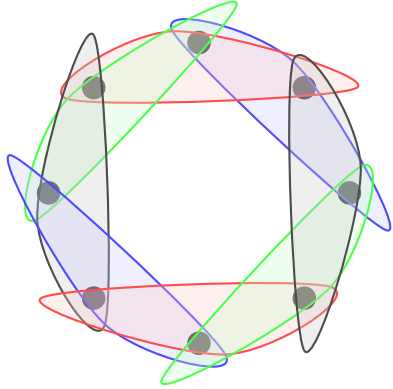
\begin{figure}[H]
\centering
\begin{tikzpicture}[scale=1,
  mystyle/.style={circle, draw, fill=black, inner sep=3pt},
  hedgeAfill/.style={fill=blue!20, opacity=0.3},
  hedgeAout/.style={draw=blue!70, thick},
  hedgeBfill/.style={fill=red!20, opacity=0.3},
  hedgeBout/.style={draw=red!70, thick},
  hedgeCfill/.style={fill=green!20, opacity=0.3},
  hedgeCout/.style={draw=green!70, thick},
  hedgeDfill/.style={fill=black!20, opacity=0.3},
  hedgeDout/.style={draw=black!70, thick}
]

% ---- vertices ----
\node[mystyle] (A) at ( 2.0,  0.0) {};
\node[mystyle] (B) at ( 1.4,  1.4) {};
\node[mystyle] (C) at ( 0.0,  2.0) {};
\node[mystyle] (D) at (-1.4,  1.4) {};
\node[mystyle] (E) at (-2.0,  0.0) {};
\node[mystyle] (F) at (-1.4, -1.4) {};
\node[mystyle] (G) at ( 0.0, -2.0) {};
\node[mystyle] (H) at ( 1.4, -1.4) {};

% ---- helper macro for soft hull coordinates ----
\newcommand{\hcoords}[3]{%
  ($(#1)!1.25!(#2)$)
  ($(#2)!1.25!(#3)$)
  ($(#3)!1.25!(#1)$)
}

% ---- helper macro to fill & outline one hyperedge ----
\newcommand{\hedge}[5]{% fillstyle, outstyle, A, B, C
  \fill[#1]  plot [smooth cycle, tension=0.6] coordinates { \hcoords{#3}{#4}{#5} };
  \draw[#2]  plot [smooth cycle, tension=0.6] coordinates { \hcoords{#3}{#4}{#5} };
}

% ---- the 8 consecutive triples (tight cycle) ----
\hedge{hedgeAfill}{hedgeAout}{A}{B}{C}
\hedge{hedgeBfill}{hedgeBout}{B}{C}{D}
\hedge{hedgeCfill}{hedgeCout}{C}{D}{E}
\hedge{hedgeDfill}{hedgeDout}{D}{E}{F}
\hedge{hedgeAfill}{hedgeAout}{E}{F}{G}
\hedge{hedgeBfill}{hedgeBout}{F}{G}{H}
\hedge{hedgeCfill}{hedgeCout}{G}{H}{A}
\hedge{hedgeDfill}{hedgeDout}{H}{A}{B}

\end{tikzpicture}
\caption{The tight cycle $C_8^{(3)}$ on 8 vertices.}
\label{fig:tight-cycle-8}
\end{figure}

\begin{corollary}\label{cor:tight cycle theorem}
    For all $t\geq 5$, we have $R(C_t^{(3)},S_k^{(3)}) = k^{O_t(1)}.$ 
\end{corollary}

To prove \cref{thm: Simple Cycles are polynomial}, we show that $\pgr$ is preserved by an operation we call bridging, which is defined by duplicating a row or column and adding a single edge between the two duplicates (see \cref{def:grid-dup} for the formal statement). We see that this operation is closely related to blowups at a vertex. Conlon, Fox, Gunby, Mubayi, Suk, Verstr\"aete, Yu, and the first author conjectured in \cite{conlon2025offdiagonalhypergraphramseynumbers} that $R(H, K_k^{(3)})$ is polynomial in $k$ if and only if $H$ is a subgraph of an iterated blowup of an edge, with the ``if" direction already known due to \cite{erdos1972ramsey}. They offer some evidence of the forward implication by showing that if $R(H, K_{k}^{(3)})$ is superpolynomial for $H$ a tightly connected 3-graph which is not bipartite. This suggests that blowup-type operations may be an enlightening framework under which to further study the polynomial to exponential transition of off-diagonal 3-uniform Ramsey numbers. Toward our result, it turns out that all simple cycles can be obtained from a single-edge graph by iterated bridging.

In the case that $H$ is a simple tree, we show that the corresponding Ramsey growth rate is linear.

\begin{theorem}\label{thm:tree linear}
For any simple tree $T$, $\gr(T,K_k)=O_T(k)$.
\end{theorem}

We next show a connection to the multicolor version of the celebrated Erd\H{o}s-Hajnal conjecture.

\begin{conj}[Multicolor Erd\H{o}s-Hajnal conjecture (see Fox, Grinspuch, Pach \cite{FOX201575})] \label{conj:meh}
     For every fixed $k, m\geq 2$ and $m$-coloring $\chi$ of $E(K_k)$, there exists $\varepsilon>0$ such that every coloring of the edges of $K_n$ contains either $k$ vertices whose edges are colored according to $\chi$ or $n^\varepsilon$ vertices whose edges are colored with at most $ m-1 $ colors. 
\end{conj}

Conditional on this conjecture, we show that many additional grid subgraphs lie in $\pgr$.

\begin{theorem} \label{thm:no horiz alignment}
    Conditional on \cref{conj:meh}, for any grid subgraph $H$ such that no two horizontal edges of $H$ are between the same two columns, then $H$ lies in $\pgr$.
\end{theorem}

As a corollary, any square-free grid subgraph on $2$ rows lies in $\pgr$. Unconditionally, we can show for such $H$ the growth rate of $\gr(H,K_k) = e^{O((\log k)^2)}$ using the upper bounds on \cref{conj:meh} due to Erd\H{o}s and Hajnal \cite{erdos_ramsey-type_1989}. Our results suggest the following tantalizing conjecture.

\begin{conj} \label{conj:squarefree}
    A grid subgraph $H$ lies in $\pgr$ if and only if $H$ is square-free.
\end{conj}

This paper is organized as follows. We define the problems formally in \cref{sec: preliminaries}. 
%Specifically, we resolve the grid Ramsey numbers of any simple tree as linear in \cref{thm:tree linear} and of $AC_6$ as $O(k^3)$ in \cref{Thm: Bounds on Ell}, and we later relate this to hypergraph Ramsey results in \cref{thm: Hypergraph Constructible by Duplication is Polynomial}. 
In \cref{sec: duplication}, we define the bridging operation, show that $\pgr$ is closed under bridging, and prove \cref{thm: Simple Cycles are polynomial}.
In \cref{sec: strong bounds}, we prove \cref{thm:tree linear} and a cubic bound on $\gr(AC_6, K_k)$.
% In \cref{thm: Simple Cycles are polynomial} and \cref{thm: Hypergraph Constructible by Duplication is Polynomial}, we see the consequences of the results in \cref{sec: duplication} on the grid Ramsey numbers of alternating and simple cycles and the hypergraph Ramsey numbers of tight cycles. In particular, these results allow us to prove \cref{cor:tight cycle theorem}. 
In \cref{sec: multicolor conjecture}, we prove \cref{thm:no horiz alignment} conditional on the multicolor Erd\H{o}s-Hajnal conjecture. Finally, \cref{sec: further directions}, we offer a few remarks and further research directions.
Finally, in \cref{sec:hypergraph} we collect the consequences of our results for hypergraph Ramsey numbers, including \cref{cor:tight cycle theorem}. We close with open problems in \cref{sec: further directions}.

\section{Preliminaries}\label{sec: preliminaries}

We define a \emph{coclique} of size $k$ in a spanning grid subgraph $H$ to be a subset $S \subseteq V(H)$ with $|S| = k$ such that $S$ forms a $k$-clique in $\bar H$. Equivalently, $S$ is an independent set of size $k$ lying within a given row or column.

\begin{definition}[Embedding]
An embedding of a grid subgraph $H$ on $c_H$ columns and $r_H$ rows into a (typically spanning) grid subgraph $G$ on $r_H$ columns and $r_G$ rows is a pair of injective functions $\varphi = (\varphi_c,\varphi_r): [c_H] \times [r_H] \rightarrow [c_G] \times [r_G]$, such that for all $u, v \in V(H)$, if $\{u, v\}$ is an edge in $H$, then $\{\varphi(u), \varphi(v)\}$ is an edge in $G$.

Sometimes, if there is an embedding map $\varphi$ which sends $H$ to a subgraph $H'\subseteq G$, we abuse notation by saying that $H'$ is an embedding of $H$." Denote by $t_g(H,G)$ the number of embeddings of $H$ in $G$.
\end{definition}

\begin{definition}[Vertical/horizontal degrees and neighborhoods]
  Given a grid subgraph $G \subseteq G_{m \times n}$ and a vertex $v = (c, r) \in V(G)$ which lives in column $c$ and row $r$, denote by $\dvert(v)$ and $\dhoriz(v)$ the vertical and horizontal degree of $v$ respectively. That is, let $\dvert(v) = d_{G[\{c\} \times [n]]} (v)$ and $\dhoriz(v) = d_{G[[m] \times \{r\}]}(v)$ where $G[\{c\} \times [n]]$ and $G[[m] \times \{r\}]$ denote the subgraphs of $G$ induced on column $c$ and row $r$ respectively.
\end{definition}

The first natural question to ask is when the grid Ramsey number exists. Let's consider 2-edge-colorings on the complete grid graph $G_{N \times N}$. Let $\mathcal G$ denote a family of grid graphs each equipped with such a coloring, and define $\gr(\mathcal G)$ to be the least $N$ such that any 2-edge-coloring of the complete grid graph $G_{N \times N}$ contains an appropriately-colored copy of some graph in $\mathcal G$ (and $\infty$ if no such $N$ exists). %This can immediately be seen to be a generalization of the definition of $\gr(G, H)$ for two graphs $G$ and $H$ by letting $\mathcal G = \{G, H\}$ colored entirely by 1 and entirely by 2 respectively.
% \jasper{Change to 1, 2-coloring so we can use i, j as subscripts}
For $i,j \in [2]$, define $\chi_{i,j}:E(G_{N\times N}) \rightarrow [2]$ to be the coloring that colors all horizontal edges $i$ and all vertical edges $j$. %We can now state the full characterization.

\begin{prop}[Existence of grid Ramsey numbers]
Let $\mathcal G$ be a family of grid graphs each equipped with a 2-edge-coloring. Then $\gr(\mathcal G) < \infty$ if and only if there exists some $N \in \N$ such that for all $i,j \in [2]$, there is some $G \in \mathcal G$ such that $\chi_{i,j}$ contains a correctly colored copy of $G$.
\end{prop}

\begin{proof}
  The necessity of such a condition is clear, as otherwise we could color an arbitrarily large grid graph such that no (colored) graph in the family $\mathcal G$ appears. We show that this condition is also sufficient by showing that for any $M \in \N$, there exists $N \in \N$ such that for any 2-edge-coloring of $G_{N \times N}$ we can always find a copy of $G_{M \times M}$ colored by one of the colorings of $\mathcal C$. That is, in a sufficiently large colored grid graph, we can find a large subgrid with identical monochromatic columns and identical monochromatic rows. 
  
  Fix such an $M$, and let $\chi: E(G_{N \times N}) \rightarrow [2]$ be an edge-coloring on the complete $N \times N$ grid graph, where $N > 2R(M, M) \cdot \binom{R(L, L)}{L}$ for some large $L \in \N$ to be determined later.
  Consider the induced subgraph on the first $R(L, L)$ rows. In each column, there exists a monochromatic clique of size $L$, and it occurs in one of $\binom{R(L, L)}{L}$ positions within the column. That is, the position of the set of vertices is uniquely given by the column index and the $L$ row indices of the vertices. Since there are $N > 2R(M, M) \cdot \binom{R(L, L)}{L}$ columns, there exist distinct rows $y_1, y_2, \dots, y_L$ such that at least $2R(M, M)$ columns have either a color 1 $L$-clique or a color 2 $L$-clique contained in these $L$ rows by the pigeonhole principle. In particular, we can find $R(M, M)$ distinct columns $x_1, x_2, \dots, x_{R(M, M)}$ which either all have a color 1 clique or all have a color 2 clique in these rows. 
  
  Denote by $G'$ the induced subgraph on rows $y_1, y_2, \dots, y_L$ and columns $x_1, x_2, \dots, x_{R(M, M)}$. Observe that $G'$ is an $R(M, M) \times L$ grid graph with all vertical edges the same color. Because each row has $R(M, M)$ vertices, we can find in each row a monochromatic $M$-clique. Taking $L > 2M \cdot \binom{R(M, M)}{M}$ and running a symmetric argument to above, we can see that $G'$ contains an $R(M, M) \times M$ grid graph where all horizontal edges are color 1 or all horizontal edges are color 2. Further restricting to any $M \times M$ subgrid gives a graph $\mathcal G_M$ as desired.  
\end{proof}

\section{Upper Bounds via Bridging}\label{sec: duplication}
In this section, we prove the main result \cref{thm: Simple Cycles are polynomial}, using a ``supersaturation" argument similar to those in Erd\H os and Simonovits \cite{erdos_supersaturated_1983} and \cite{Fox_2021} adapted to the grid setting. In fact, we prove a significant generalization of \cref{thm: Simple Cycles are polynomial}, in the following form.

%\xh{We need to come up with a word for duplication + adding an edge, as such it is very susceptible to being misunderstood to not involve adding the extra edge on fast reading. Maybe "bridging"?}
%\jasper{I made the change to bridging}

\begin{definition}[Row/column-bridging]\label{def:grid-dup}
  Let $H = (V_H,E_H)$ be a grid subgraph on $c$ columns and $r$ rows. \textit{Bridging of column $c_* \in [c]$ at a row $y \in [r]$} yields a grid subgraph $H'$ on $c+1$ columns and $r$ rows with the following edges: for $a,d \in [c+1]$ and $b,e \in [r]$,
  \begin{itemize}
  \item $\{(a,b),(d,e)\}$ if $\{(a,b),(d,e)\}\in E_H$ and $a,d \neq c+1$
  \item $\{ (a,b),(c+1,b) \}$ if $\{(a,b),(c_*,b) \} \in E_H$
  \item $\{ (c+1,b),(c+1,e) \}$ if $\{ (c_*,b),(c_*,e) \} \in E_H$
  \item $\{ (c_*,y),(c+1,y) \}$
  \end{itemize}
  Note that by relabeling columns, we may duplicate column $c_*$ into any column index in $[c+1]$, not just $c+1$. By transposition of horizontal and vertical edges,  bridging of row $r_* \in [r]$ in column $x\in [c]$ is defined analogously.
\end{definition}

\begin{theorem}[Row/column-bridging on grid subgraphs]\label{Cor: col/row Dup GridGraph}
 Let $G$ be a spanning grid subgraph on $N$ rows and $N$ columns such that $G$ does not contain a $k$-coclique, and let $H$ be a grid subgraph on $r$ rows and $c$ columns. Finally, consider $H'$ formed by row or column-bridging of $H$. Suppose that there exists constants $\alpha,\beta,\beta' > 0$ such that if $N > \alpha k^\beta$, then $t_g(H,G) > \alpha^{-1} N^{r+c}k^{-\beta'}$. Then, if $N > 2\alpha k^{\beta'+1}$, then $t_g(H',G) > (2\alpha)^{-2}N^{r+c+1}k^{-(2\beta'+1)}$.
\end{theorem}

\begin{proof}
Assume $n >2\alpha k^{c'+1}$. Without loss of generality, suppose $H'$ is achieved from $H$ by bridging column $c$ and adding the edge $\{(c,y),(c+1,y)\}$. Let $H^-$ be $H$ restricted to columns in $[c-1]$.

For any embedding $(\varphi_c,\varphi_r):H^-\to V$, define $P_{\varphi_c,\varphi_r}$ as the set of columns $u \in [N]\setminus \rm{Im}(\varphi_c)$ such that $(\varphi_c,\varphi_r)$ can be extended to an embedding of $H$ by also sending column $c$ to column $u$. Thus, by assumption,
\[\sum_{\varphi_c,\varphi_r} |P_{\varphi
_c,\varphi_r}| =t_g(H,G)> \frac{N^{r+c}}{\alpha k^{\beta'}}.\]
Now, note that for any given $(\varphi_c,\varphi_r)$ and a pair $u,u' \in P_{\varphi_c,\varphi_r}$, an embedding of $H'$ is formed if we have the edge $\{(u,\varphi_r(y)),(u',\varphi_r(y))\}\in E(G)$. Let $F_{\varphi_c,\varphi_r,y}$ be the graph on the elements of $P_{\varphi_c,\varphi_r}$ such that $\{u,u'\}\in F_{\varphi_c,\varphi_r,y}$ if and only if $\{ (u,\varphi(y)),(u',\varphi(y))\} \in E$. Then, each edge in $F_{\varphi_c,\varphi_r,y}$ forms an embedding of $H'$. Note that $F_{\varphi_c,\varphi_r,y}$ must not contain an independent set $U$ of size $k$, or else $E(G)$ contains a horizontal coclique of size $k$. If we define $f(x)=\frac{x(x-k+1)}{2(k-1)}$, then for any given $(\varphi_c,\varphi_r)$, Tur\'an's theorem guarantees that there exists at least $f(|P_{\varphi_c,\varphi_r}|)$ embeddings of $H'$ in $G$ where the embedding of $H^-\subset H'$ in $G$ is given by $(\varphi_c,\varphi_r)$. Note that $f(x)$ is convex and $f(x) \geq x^2/(4k)$ for $x \geq 2k$. Since there are $M:=\binom{N}{c-1}\binom{N}{r} \leq N^{r+c-1}$ choices of $(\varphi_c,\varphi_r)$, the average cardinality of $P_{\varphi_c,\varphi_r}$ is greater than $\alpha^{-1} \cdot N k^{-\beta'}>2k$. Therefore, by Jensen's inequality,
\[t_g(H',G) \geq \sum_{(\varphi_c,\varphi_r)} f(|P_{\varphi_c,\varphi_r}|) \ge
Mf\Big(\frac{1}{M}\sum_{(\varphi_c,\varphi_r)}|P_{(\varphi_c,\varphi_r)}|\Big)\geq
\frac{M}{4k}\Big(\frac{N^{r+c}}{M\alpha k^{\beta'}}\Big)^2
\geq (2\alpha)^{-2}N^{r+c+1}k^{-2\beta'-1},
%\binom{n}{m-1}(m-1)! \cdot \frac{\left( \frac{n}{a \cdot k^{c'}} \right) \left( \frac{n}{a \cdot k^{c'}} - k \right)}{2(k-1)} \geq \binom{n}{m-1}(m-1)! \cdot \frac{\left( \frac{n}{a \cdot k^{c'}} \right)^2}{4k}
\]
achieving our desired result.
\end{proof}

\begin{theorem}\label{thm: Constructible by Duplication is Polynomial}
If $H$ is a grid subgraph on $r$ rows and $c$ columns which can be obtained by finite iteration of row/column-bridging operations starting from the graph on one vertex, $\gr(H,K_k)$ is $k^{O(2^{r+c})}$, explicitly at most $2^{2^{r+c}-2}k^{2^{r+c-3}}$.
\end{theorem}

\begin{proof}
    We prove this by induction using \cref{Cor: col/row Dup GridGraph} that if $G$ is a grid subgraph on $N$ rows and $N$ columns such that $G$ does not contain a $k$-coclique and $N>2^{2^{r+c}-2}k^{2^{r+c-3}}$, then $t_g(H,G)>2^{-2^{r+c}-2}N^{r+c}k^{-2^{r+c-2}+1}$. This will imply $t_g(H,G)>0$ for all such $G$, so $\gr(H,K_k)\le 2^{2^{r+c}-2}k^{2^{r+c-3}}$.
    
    Let $H$ be a single vertex, so $r=c=1$. Then if $N>2^{2^{r+c}-2}k^{2^{r+c-3}}>0$, then $t_g(H,G)>N^2/4=2^{-2^{r+c}-2}N^{r+c}k^{-2^{r+c-2}+1}$, proving the base case.

    If $H$ satisfies the inductive hypothesis and $H'$
    is obtained by a bridging operation from $H$,
    then using \cref{Cor: col/row Dup GridGraph} with $\alpha=2^{2^{r+c}-2}$, $\beta=2^{r+c-3}$, $\beta'=2^{r+c-2}-1$ gives that for $N>2\alpha k^{\beta'+1}>4\alpha^2k^{\beta'+1}=2^{2^{r+c+1}-2}k^{2^{r+c-2}}$, we have that $t_g(H',G)>(4\alpha^2)^{-1}N^{r+c+1}k^{-(2\beta'+1)}=2^{-2^{r+c}-2}k^{-2^{r+c-2}+1}$, as desired.
\end{proof}

We note the condition of \cref{Cor: col/row Dup GridGraph} is not stronger than the condition $\gr(H,K_k)=k^{O_H(1)}$.

\begin{prop}\label{prop: supersaturation}
  Suppose that $H\subseteq G_{c\times r}$, and $\gr(H,K_k)=k^{O(1)}$. Let $G$ be an $N\times N$ spanning grid subgraph with no coclique of size $k$. Then there exist constants $\alpha,\beta,\beta'$ such that if $N>\alpha k^\beta$, $t_g(H,G)>\alpha^{-1}N^{r+c}k^{-\beta'}$.
\end{prop}
\begin{proof}
Let $M$ be such that $\gr(H,K_k)<k^M$. Then in every choice of $k^M$ rows and $k^M$ columns, there is a copy of $H$. Any copy of $H$ in $G$ is part of $\binom{N-r}{k^M-r}\binom{N-c}{k^M-c}$ choices of $k^M$ rows and $k^M$ columns. Thus, the number of copies of $H$ in $G$ is at least 
\begin{align*}
  \left(\binom{N}{k^M}\binom{N}{k^M}\right)/\left(\binom{N-r}{k^M-r}\binom{N-c}{k^M-c}\right)&=\frac{N(N-1)\cdots (N-r+1)\cdot N(N-1)\cdots (N-c+1)}{k^M(k^M-1)\dots (k^M-r+1)\cdot k^M(k^M-1)\cdots (k^M-c+1)} \\
  &>\alpha^{-1}N^{r+c}k^{-M(r+c)}
\end{align*}
for some $\alpha$ if $N>k^\beta$ for some $\beta$, as desired.
\end{proof}

\begin{definition}[Alternating cycle]
  For $t$ even, an alternating cycle $\AC_t$ on $t$ vertices is the grid subgraph on $t/2$ rows and $t/2$ columns which is the cycle on the vertices $(1,1),(2,1),(2,2),(3,2),(3,3),\dots,(t/2,t/2),(1,t/2)$ in  order. See \cref{fig:AC_8} for an example.
\end{definition}

Alternating cycles have even length, in particular at least four. Note that $G_{2 \times 2} = AC_4$, so \cite{conlon2023hypergraph} showed that $\gr(AC_4,K_k)$ is superpolynomial in $k$. We prove the following grid Ramsey result classifying the polynomial versus superpolynomial divide for $AC_t$ in grid graphs against $K_k$.

\begin{theorem}\label{Thm:Alternating cycles are polynomial}
  For even $t \geq 6$, $\gr(\AC_t,K_k) = k^{O_t(1)}$.
\end{theorem}

We now prove \cref{Thm:Alternating cycles are polynomial} after the following lemma.

\begin{lemma}\label{lem: \AC_t+2 from \AC_t}
  For $t \geq 3$, $\AC_{2t+2}$ can be obtained from $\AC_{2t}$ by two bridgings.
\end{lemma}
\begin{proof}
  Given $\AC_{2t}$ as a grid subgraph on $t$ rows and $t$ columns which is the cycle on the vertices $(1,1),(2,1),\allowbreak (2,2),(3,2),(3,3),\dots,(t,t-1),(t,t),(1,t)$ in that order, duplicating row $t$ and adding edge $\{(t,t),(t,t+1)\}$ yields a grid subgraph on $t+1$ rows and $t$ columns that contains the cycle on the vertices $(1,1),(2,1),(2,2),(3,2),\allowbreak(3,3),\dots,(t,t-1),(t,t),(t,t+1),(1,t+1)$ in that order. Now, if we bridge column $t$ and add edge $\{(t+1,t),(t+1,t+1)\}$, we get a grid subgraph on $t+1$ rows and $t+1$ columns that contains the cycle on the vertices $(1,1),(2,1),(2,2),(3,2),(3,3),\dots,(t+1,t),(t+1,t+1),(1,t+1)$ in that order, which is $\AC_{2t+2}$ as desired.
\end{proof}

Now, we are ready to prove \cref{Thm:Alternating cycles are polynomial}.
\begin{proof}[Proof of \cref{Thm:Alternating cycles are polynomial}]
  By \cref{thm: Constructible by Duplication is Polynomial}, it suffices to show that $\AC_t$ is constructible by finite iteration of row/column-bridging from the graph of one vertex.\\

  By row-bridging on a single vertex, we obtain a grid subgraph on $2$ rows and $1$ column with an edge. Then, column-bridging and the addition of edge $\{(1,1),(2,1)\}$ yields a grid subgraph on $2$ rows and $2$ columns with edges $\{(1,2),(1,1)\},\{(1,1),(2,1)\},\text{ and }\{(2,1),(2,2) \}$. Subsequently, bridging row $1$ and adding edge $\{(1,1),(1,3)\}$ yields a grid subgraph on $3$ rows and $2$ columns which contains the cycle on the vertices $(1,1),(2,1),(2,2),(2,3), (1,3)$ in that order. Then, bridging column $2$ and adding edge $\{(2,2),(3,2)\}$ yields a grid subgraph on $3$ rows and $3$ columns which contains the cycle on the vertices $(1,1),\allowbreak(2,1),\allowbreak(2,2),\allowbreak(3,2),\allowbreak(3,3),\allowbreak(3,1)$ in that order, which is $\AC_6$. Therefore, by iterating bridging four times, we obtain $\AC_6$ from the graph of one vertex.

  By applying \cref{lem: \AC_t+2 from \AC_t} $(t-6)/2$ times, we see that $\AC_t$ can be created by iterating bridging $t-2 < \infty$ times from the graph of one vertex, as desired.
\end{proof}

To demonstrate one consequence of \cref{thm: Constructible by Duplication is Polynomial}, we now show that $\pgr$ is closed under \emph{generalized subdivisions}, where we replace an edge $uv$ with any connected graph that lies solely within the row or column of $uv$.

\begin{definition}[Generalized subdivision]\label{def: generalized subdivision}
  Let $H$ be a grid subgraph on $c$ columns and $r$ rows with vertex set $V(H) \subseteq [c] \times [r]$, and let $e = \{(x, y),(x', y)\} \in E(H)$. A grid subgraph $H'$ on $c + m$ columns (for some $m \geq 0$) and $r$ rows is \textit{a generalized subdivision of $H$ at $e$} if
  \begin{itemize}
    \item $V(H') = V(H) \cup V^*$ where $V^* = \{(i, y): c < i \leq c + m\}$
    \item $E(H) \subseteq E(H')$
    \item $H'[V^* \cup \{(x, y), (x', y)\}]$ is a connected graph
    \item For all $v \in V^*$, $\nhoriz(v) \subseteq V^* \cup \{(x, y), (x', y)\}$.
  \end{itemize}
  % \lp{Perhaps the itemize isn't needed.}\ghaura{I think we should stick with the older(currentlly crossed out) defintion since the new one is not correct. }
  
  If $e$ is instead a vertical edge (i.e. $e = \{(x, y),(x, y') \}$), we define a generalized subdivision analogously with $V(H') \subseteq [c] \times [r + m]$.
\end{definition}

We proceed with \cref{lem: Generalized Subdivision}, which states that $\pgr$ is closed under generalized subdivisions. %It may be noticed that depending on the choice of $H$, the repeated application of \cref{thm: Constructible by Duplication is Polynomial} as in the following proof can produce a richer grid subgraph than is given by \cref{def: generalized subdivision}. Nevertheless, the following result is included to demonstrate one application of \cref{Cor: col/row Dup GridGraph}, which can be regarded as a strong technique for further study in this area.

\begin{lemma}[Generalized subdivision]\label{lem: Generalized Subdivision}
  Let $H$ be a grid subgraph on $c$ columns and $r$ rows with $\gr(H,K_k) = k^{O_H(1)}$. Further suppose $e \in E(H)$ and $H'$ is a grid subgraph formed by generalized subdivision of $H$ at $e$. Then $\gr(H', K_k) = k^{O_{H'}(1)}$.
\end{lemma}

% \jasper{The only "issue" here is that I don't think we explicitly state that row/column duplication of a pattern with poly gr number gives another pattern with poly gr number. It's essentially given by \cref{Cor: col/row Dup GridGraph}, but maybe we need to explicitly state that poly gr bound implies the hypotheses are met. Or maybe we need to put this in the statement of \cref{thm: Constructible by Duplication is Polynomial}, i.e. we need not start with a single vertex and can instead start with anything polynomial}
\begin{proof}
  Assume without loss of generality that $e = \{(x, y),(x', y)\}$ is a horizontal edge and that the subdivision of $e$ to form $H'$ includes the addition of $m \geq 1$ vertices in row $y$ (note $m = 0$ is a trivial subdivision and implies $H' = H$). 
  % We claim that there exists a grid pattern $H^* \supseteq H'$ which can be constructed by $m$ row bridging steps from $H$.
  We claim that there exists a sequence of grid subgraphs $H =: H_0^* \subseteq H_1^* \subseteq \cdots \subseteq H_m^* = :H^*$ formed by $m$ row-bridging steps such that $H^* \supseteq H'$. First, we form $H_1^*$ by bridging column $x$ at row $y$ to form column $c + 1$ and add edge $\{(x, y),(c + 1, y)\}$. For $2 \leq i \leq m$, we form $H_i^*$ by bridging column $c + i$ at row $y$ to form column $c + i + 1$ and add the edge $\{(c + i, y),(c + i + 1, y)\}$.

  It remains to show that $H' \subseteq H^*$. Clearly $V(H^*) \supseteq V(H)$ and $E(H^*) \supseteq E(H)$. Taking $V^*$ as defined in \cref{def: generalized subdivision}, it remains to check that $H'[V^* \cup \{(x, y), (x', y)\}] \subseteq H^*[V^* \cup \{(x, y), (x', y)\}]$. Here we make the observation that $H^*[V^* \cup \{(x, y), (x', y)\}] \cong K_{m + 2}$. Moreover, it can be readily seen by induction that $\tilde H_i^* := H_i^*[\{(x, y), (x', y), (c + 1, y), \dots, (c + i, y)\}] \cong K_{i + 1}$. The base case is clear as for $H_0^* = H$ we have by assumption that $e = \{(x, y), (x', y)\}$ is an edge in the graph. Now, suppose the statement holds for some $0 \leq i \leq m - 1$ and consider the induced subgraph $\tilde H_{i+ 1}^* = H_{i + 1}^*[\{(x, y), (x', y), (c + 1, y), \dots, (c + i + 1, y)\}]$. By the inductive hypothesis, it suffices to check that the vertex $(c + i + 1, y)$ is adjacent to all other vertices in $\tilde H_{i+1}^*$. As this vertex was formed by bridging column $c + i$, we have $N_{\tilde H_{i + 1}^*}(c + i + 1, y) \supseteq N_{\tilde H_i^*}(c + i, y) = \{(x, y), (x', y), (c + 1, y), \dots, (c + i - 1, y)\}$. Finally, we have $\{(c + i, y), (c + i + 1, y)\} \in E(\tilde H_{i + 1}^*)$ by the bridging at row $y$. Then as $H^*$ is formed by finitely many row-bridging operations, the proof is complete by \cref{thm: Constructible by Duplication is Polynomial}.
  % To see this, consider the new columns in order. We have the edge $\{(c + 1, y), (x', y)\}$ because $H$ contains the edge $\{(x, y), (x', y)\}$ and column $c + 1$ was formed by bridging of column $x$. In addition, we have the edge $\{(c + 1, y), (x, y)\}$ because we added this edge in the bridging step. Now, for $1 < i \leq m$, we similarly have $\nhoriz((c + i, y)) \supseteq \nhoriz((c + i - 1, y))$ by the bridging of column $c + i - 1$ to form column $c + i$, and we have $(c + i - 1, y) \in \nhoriz((c + i, y)$ by the added edge. Then for all $1 \leq i < j \leq m$, we have that $(x, y), (x', y), (c + i, y) \in \nhoriz((c + j, y))$, hence $H^*[V^* \cup \{(x, y), (x', y)\}] \cong K_{m + 2}$ as claimed. Then as $H' \subseteq H^*$, we have $\gr(H', K_k) \leq \gr(H^*, K_k)$. As $H^*$ was formed by finitely many row-bridgings from $H$, \cref{Cor: col/row Dup GridGraph} implies a polynomial (in $k$) upper bound on $\gr(H^*, K_k)$.
\end{proof}
%\lp{This is correct, but could be neater? Especially the last paragraph is informal. 1) Induction on $m$ would be clean. 2) If you don't want induction, the repeated bridging forms a sequence $H^*_1,\ldots, H^*_m$, and you can speak on these concretely.}
%\jasper{I tried to clean up above. It looks good to me now but I'd appreciate another set of eyes to check for correctness. Also, if someone can suggest a better notation for where I used $\tilde H_i^*$ that doesn't mess up the spacing, that might help. I figured we should give it some notation because writing out the definition in the subscripts for the neighborhoods would be too cumbersome and look odd.}

With generalized subdivision and further use of \cref{thm: Constructible by Duplication is Polynomial}, we can now prove \cref{thm: Simple Cycles are polynomial}.

\begin{proof}[Proof of \cref{thm: Simple Cycles are polynomial}]
  Let $H$ be a simple cycle on vertices $(x_1,y_1),(x_2,y_2),\dots,(x_\ell,y_\ell)$ in that order. Then, by definition of simple cycle, we can find $i_1,i_2,\dots,i_m \in [\ell]$ where $(x_{i_1},y_{i_1}),(x_{i_2},y_{i_2}),\dots,(x_{i_m},y_{i_m})$ forms the alternating cycle $AC_{m}$ and for each $j \in [\ell]$, either $j \in \{ i_1,i_2,\dots,i_m \}$, $x_j = x_{i_\rho} = x_{i_{\rho+1}}$ for some $\rho \in [m]$ with $m+1 := 1$, or $y_j = y_{i_\rho} = x_{i_{\rho+1}}$ for some $\rho \in [m]$ with $m+1 := 1$. 
  Thus, $H$ can be obtained from $AC_m$ by a series of generalized edge subdivisions, and since $\gr(AC_m, K_k) = k^{O_m(1)}$ by \cref{Thm:Alternating cycles are polynomial}, then \cref{lem: Generalized Subdivision} implies $\gr(H,K_k) = k^{O_H(1)}$.
\end{proof}

\section{Simple Trees and $\AC_6$}\label{sec: strong bounds}

%\xh{Change tree and grid tree to "simple tree" throughout}

In this section, we show linear bounds on the grid Ramsey number of simple trees, and cubic bounds on the grid Ramsey number of $AC_6$. It would be interesting to understand the polynomial order of $\gr(H, K_k)$ for other $H$ in $\pgr$. %Further, we explore the related results on loose trees and the tight cycle of length $6$ in hypergraph Ramsey theory.

We first prove \cref{thm:tree linear}. To do this, we first define the notion of an $n$-diverse vertex for a given tree type $T$ in a grid subgraph $G$. Roughly speaking, a vertex $v$ is $n$-diverse if there are $n$ copies of $T$ containing $v$, mutually disjoint except at $v$.
\begin{definition}[$n$-diverse]
Let $T$ be a simple tree, $v=(s, t)$ be a vertex of $T$, and $n \geq$ be an integer. In an $N\times N$ grid subgraph $G$, a vertex $(k, \ell)$ of $G$ is \emph{$n$-diverse} for $v\in T$ if there exist embeddings $\varphi_i = (\varphi_c^{(i)},\varphi_r^{(i)}), 1 \leq i \leq n$ from $T$ to $G$ such that
% \\\lp{recall $(\phi,\varphi)$ discussion from meeting}
\begin{itemize}
\item $\varphi_c^{(i)}(s)=k,\varphi_r^{(i)}(t)=\ell$ for all $i$,
\item For all $i_1 \neq i_2$ and all $j_1, j_2$, $\varphi_c^{(i_1)}(j_1) = \varphi_c^{(i_2)}(j_2) \implies j_1 = j_2 = s$, and
\item For all $i_1 \neq i_2$ and all $j_1, j_2$, $\varphi_r^{(i_1)}(j_1) = \varphi_r^{(i_2)}(j_2) \implies j_1 = j_2 = t$.
% \item for all $n_1\neq n_2$ and all $i_1\lps{}{\neq}i_2$, $\phi_{n_1}(i_1)\lps{}{\neq}\phi_{n_2}(i_2)$ \lps{iff $i_1=i_2=i$}{}, and
% \item for all $n_1\neq n_2$ and all $j_1\lps{}{\neq}j_2$, $\varphi_{n_1}(j_1)\lps{}{\neq} \varphi_{n_2}(j_2)$ \lps{iff $j_1=j_2=j$}{}.
\end{itemize}
When $v,T$ are clear we simply write that $(s,t)$ is $n$-diverse.
\end{definition}

\cref{thm:tree linear} is implied by the following lemma when $n=1$ and $v$ is any vertex of $T$.
\begin{lemma}
Let $c$ be a constant, $T$ be a simple tree, and $v$ be a vertex of $T$.
Consider an $N\times N$ grid subgraph $G$ with no coclique of size $k$. Then there is a constant $c'$ dependent on $c,v,T$ such that at least $N^2-c'Nk$ vertices of $G$ are $c$-diverse for $v\in T$.
\end{lemma}
\begin{proof}
We induct on $|T|$, with base case being $|T|=1$. In this case, every vertex of $G$ is $c$-diverse for any $c$ by letting all embeddings be the one that sends $T$ to that vertex. Hence, $c'=0$ works.

For the inductive step, let $v_2$ be a vertex of $T$ adjacent to $v$; without loss of generality we assume the edge $\{v, v_2\}$ is horizontal. Then deleting the edge $\{v,v_2\}$ gives a simple tree $T_1$ with $v$ and a simple tree $T_2$ with $v_2$ such that $T_1\cup T_2=T-\{v,v_2\}$. We look for vertices in $G$ which
\begin{itemize}
\item are $100c|T|$-diverse for $v\in T_1$ and $v_2\in T_2$, and
\item have horizontal degree at least $c$ to other vertices satisfying the first bullet. 
\end{itemize}
By the inductive hypothesis, there are constants $c_1,c_2$ such that $N^2-c_1Nk$ vertices are $100c|T|$ diverse for $v\in T_1$ and $N^2-c_2Nk$ vertices are $100c|T|$-diverse for $v_2\in T_2$. Thus, at least $N^2-(c_1+c_2)Nk$ vertices are $100c|T|$ diverse for both.

In the induced subgraph of $G$ with only vertices that are $100c|T|$ diverse for both $v\in T_1$ and $v\in T_2$, there are at most $ck$ vertices per row which have horizontal degree less than $c$. This is because if there are $ck$ vertices in a row with horizontal degree less than $c$, Tur\'an's theorem implies there is a coclique of size $k$ among these $ck$ vertices. Hence, there are at least $N^2-(c_1+c_2+c)Nk$ vertices satisfying the conditions above.

Now we show any vertex satisfying the two bulleted properties is $c$-diverse for $v\in T$. Let $w\in G$ be such a vertex, with horizontal neighbors $u_1,u_2,\dots, u_c$ which are also $100c|T|$-diverse for $v\in T_1$ and $v_2\in T_2$. Let $i$ be the largest integer ($1$ is allowed) for which there are embeddings $T^1,T^2,\dots, T^{i-1}$ of $T$ such that
\begin{itemize}
  \item for all $j<i$, $v,v_2$ are sent to $w,u_j$ respectively in the embedding that sends $T$ to $T^j$,
  \item if $j\neq \ell$ then $T^j$ and $T^\ell$ share no rows/columns except for the row/column of $w$, and
  \item for all $j<i$ and $\ell\le c$, $T^j$ and $u_\ell$ share no rows/columns except for the row of $w$.
\end{itemize}
%\lp{Don't use index $k$, since $k$ is already used.}

If $i=c+1$ then we are done. If $i\le c$, we can show a contradiction. Because $w$ is $100c|T|$-diverse for $v\in T_1$, off the $100c|T|$ embeddings that make $w$ $100c|T|$-diverse, except for the row/column of $w$, every row/column of $G$ has part of at most $1$ embedding. Since $T^1,\dots, T^{i-1},u_i,u_{i+1},\dots, u_c$ inhabit at most $2c|T|$ rows/columns in total, we can find an embedding $T^{i}_1$ of $T_1$ where $v\in T_1$ is sent to $w$ such that $T^i_1$ shares no row/column with the other $T^j$ or $u_\ell$ except for the row/column of $w$. We can similarly find an embedding $T^i_2$ of $T_2$ where $v_2\in T_2$ is sent to $u_i$ such that $T^i_2$ shares no row/column with $T^i_1$, $T^j$ for $j<i$, and $u_\ell$ for $\ell\neq i$ except for the row of $w$. Then letting $T^i$ be the union of $T^i_1$, $T^i_2$, and the edge from $w$ to $u_i$ suffices.
\end{proof}

For our final unconditional result, we prove polynomial bounds on the grid Ramsey number of $\AC_6$, which is not linear. Write $R(3,k)$ for the classical graph Ramsey number of a triangle against a $k$-clique, which is known to be of order $\Theta(k^2/ \log k)$, by Kim \cite{Kim1995R3t} and Ajtai-Koml\'os-Szemer\'edi \cite{ajtai_note_1980} and had lower bound recently improved to $(\frac{1}{3} + o(1))\frac{k^2}{\log k}$ and subsequently $(\frac{1}{2} + o(1))\frac{k^2}{\log k}$ by Campos, Jenssen, Michelen, and Sahasrabudhe \cite{campos2023newlowerboundsphere} and Hefty, Horn, King, and Pfender \cite{hefty2025improvingr3kjustbites}.

\begin{theorem}\label{Thm: Bounds on Ell}
$R(3,k)\le \gr(\AC_6,K_k)\le c'k^3$ for some constant $c'$.
\end{theorem}
\begin{proof}
For the lower bound, let $N=R(3,k)-1$. There is a graph $G$ on $N$ vertices which has no independent set of size $k$ and no triangle. The grid subgraph $G\square K_{N}$ is a grid subgraph on $N$ rows and $N$ columns, we can show it has no $\AC_6$ and no coclique of size $k$.

Suppose $(\varphi_c,\varphi_r)$ is an embedding of $\AC_6$ into $G$. Then there is an edge $\{(\varphi_c(1),\varphi_r(1)),(\varphi_c(2),\varphi_r(1)\}$. Thus, in $G$ the edge $\{\varphi_c(1),\varphi_c(2)\}$ exists. Similarly $\{(\varphi_c(2),\varphi_c(3))\}$ and $\{(\varphi_c(3),\varphi_c(1))\}$ also exist in $G$, so $G$ has a triangle. Thus, $G\square K_{N}$ has no $\AC_6$.

We have that $G\square K_N$ has no horizontal coclique of size $k$ because then $G$ would also have a horizontal coclique of size $k$. We also have that $G\square K_N$ has no vertical coclique of size $k$ because all vertical edges are present, proving the lower bound.

For the upper bound, let $N=ck^3$ for some sufficiently large constant $c$, and suppose there is an $N\times N$ grid subgraph with no coclique of size $k$. Then in each row, consider the $\frac{N}{3}$ vertices with lowest degree. By Tur\'an's theorem on these vertices, there is a vertex among these with horizontal degree at least $\frac{N}{3k}$. Thus, $\frac{2N}{3}$ of the vertices in any row have horizontal degree at least $\frac{N}{3k}$. Similarly, $\frac{2N}{3}$ of the vertices in any column have vertical degree at least $\frac{N}{3k}$. By the union bound, at least $\frac{N^2}{3}$ vertices have both horizontal degree and vertical degree at least $\frac{N}{3k}$. Thus, there are at least $\frac{N^4}{27k^2}$ ordered quadruples $(x_1,y_1,x,y)$ where the edges $\{(x_1,y_1),(x_1,y)\}$ and $\{(x_1,y_1),(x,y_1)\}$ exist.

By the pigeonhole principle, there exists a row $y$ and column $x$ with $\frac{N^2}{27k^2}$ vertices $(x_1,y_1)$ with the edges $\{(x_1,y_1),(x_1,y)\}$ and $\{(x_1,y_1),(x,y_1)\}$ present. Call these \emph{corners}. Between two corners $(x_1,y_1)$ and $(x_2,y_1)$ in the same row, draw an ``auxiliary edge" between them if the edge $\{(x_1,y),(x_2,y)\}$ exists. Similarly, for two corners $(x_1,y_1)$ and $(x_1,y_2)$ in the same column, draw an auxiliary edge between them if the edge $\{(c,r_1),(c,r_2)\}$ exists.

In a row $y_1$, suppose $k$ of its corners, $(x_1,y_1),(x_2,y_1),\dots, (x_k,y_1)$ don't have any horizontal auxiliary edges incident to them. Then we obtain a coclique of size $k$ from the vertices $(x_i,y)$. Thus, less than $k$ corners in each row have horizontal auxiliary degree $0$, so at most $N(k-1)$ corners have no horizontal auxiliary edges. By repeating this argument for the columns, we get at most $2N(k-1)$ corners have no horizontal auxiliary edge or no vertical auxiliary edge. Since $c'$ is large enough, we have $\frac{N^2}{27k^2}>2N(k-1)$, so there exists a corner $(x_1,y_1)$ with positive horizontal and vertical auxiliary degree. Thus, there are $x_2$ and $y_2$ such that $(x_2,y_1)$ and $(x_1,y_2)$ are corners and the edges $\{(x_1,y),(x_2,y)\}$ and $\{(x,y_1),(x,y_2)\}$ exist. We obtain an $\AC_6$ into our grid from the embedding $(\varphi_c,\varphi_r)$ where $\varphi_c(1)=x_1,\varphi(2)_c=x_2,\varphi_c(3)=x$ and $\varphi_r(1)=y,\varphi_r(2)=y_1,\varphi_r(3)=y_2$, as desired.
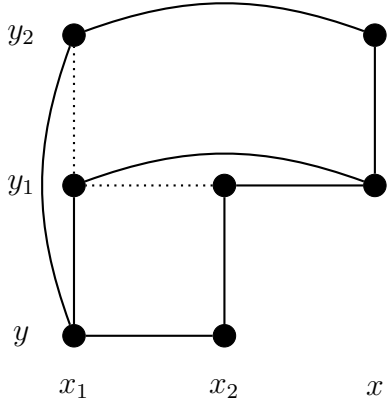
\begin{figure}[H]
 \centering
\begin{tikzpicture}[scale=2, mystyle/.style={circle, draw, fill=black, inner sep=3pt}]
 % Nodes
 \node[mystyle] (A) at (1,0) {};
 \node[mystyle] (B) at (0,0) {};
 \node[mystyle] (C) at (2,1) {};
 \node[mystyle] (D) at (1,1) {};
 \node[mystyle] (E) at (0,1) {};
 \node[mystyle] (F) at (2,2) {};
 \node[mystyle] (G) at (0,2) {};
 \node (col1) at (0,-0.35) {\large $x_1$};
 \node (col1) at (1,-0.35) {\large $x_2$};
 \node (col1) at (2,-0.35) {\large $x$};
 \node (col1) at (-0.35,2) {\large $y_2$};
 \node (col1) at (-0.35,1) {\large $y_1$};
 \node (col1) at (-0.35,0) {\large $y$};

 % Edges – modify or add as needed
 \draw[black,thick] (A) -- (B);
 \draw[black,thick] (C) -- (D);
 \draw[black,thick] (C) to[out=-200,in=20] (E);
 \draw[black,thick,dotted] (D) -- (E);
 \draw[black,thick] (F) to[out=-200,in=20] (G);
 \draw[black,thick] (F) -- (C);
 \draw[black,thick] (D) -- (A);
 \draw[black,thick,dotted] (E)--(G);
 \draw[black,thick] (B) -- (E);
 \draw[black,thick] (B) to[out=110, in=250] (G);
\end{tikzpicture}
\label{fig:findell}
\caption{Dotted edges denote auxiliary edges.}
\end{figure}
\end{proof}

\section{Relation to Multicolor Erd\H{o}s-Hajnal Conjecture} \label{sec: multicolor conjecture}

All theorems in this section are dependent on the Multicolor Erd\H{o}s-Hajnal conjecture, \cref{conj:meh}. The conjecture is well-known to have implications to Ramsey theory; for example Conlon, Fox, and R\"odl \cite{conlon2015hedgehogsarenotcolorblind} showed that a particular case is directly equivalent to the existence of a polynomial upper bound for the 3-color Ramsey number of the Hedghehog graph $R_3(H_3; 3)$. While some partial results are known (see for example section 3.3 of Conlon, Fox, Sudakov \cite{conlon2015recentdevelopments} for a known case in which the result holds), the general case remains open.

\begin{lemma}[Multicolor Erd\H{o}s-Hajnal, no aligning edges]\label{MEH lemma}
Assume the conclusion of the Multicolor Erd\H{o}s-Hajnal conjecture.
Let $c$ and $r$ be two fixed positive integers. Let $H$ be a fixed grid subgraph on $c$ columns and $r$ rows of only horizontal edges such that no two edges are between the same two columns. Then there exists an $M$ such that if $N>k^M$, there exists either a horizontal coclique of size $k$ or a copy of $H$ in any $r\times N$ grid subgraph $G$.
\end{lemma}

\begin{proof}
Consider the auxiliary graph $G^* \cong K_N$ on $N$ vertices where each vertex corresponds to a column of $G$. Color the edges $\{x_1, x_2\} \in E(G^*)$ by $\chi_G(x_1,x_2)= \{y|\{(x_1,y),(x_2,y)\}\in E(G)\}$. There are $2^r$ possible colors as we allow any subset of $[r]$. Let $H^*$ be the analogous graph for $H$ and let $\chi_H$ be the analogous coloring. In $H^*$, all colors have $0$ or $1$ element.

If there is an injection $\varphi_c$ from $V(H^*)$ to $V(G^*)$ such that for each edge $\{x_1,x_2\}$ in $H$, $\chi_H(x_1,x_2)\subseteq \chi_G(\phi_c(x_1),\phi_c(x_2))$, then we get an embedding of $H$ into $G$ by $(\rm{id}_{[r]},\varphi_c)$. Assume for the sake of contradiction no such $\varphi_c$ exists. Furthermore, assume that $M$ is sufficiently large with respect to $H$ (specified concretely later) and there is no horizontal coclique of size $k$ in $G$. This means that for every $y$ and every subset $U\subset[N]$ of size $k$, there exists $x_1,x_2\in U$ such that $y\in\chi_G(x_1,x_2)$.

Let $\mathcal S$ be any subset of $2^{[r]}$ such that $\bigcup_{S \in \mathcal S}  S=[r]$. Call such an $\mathcal{S}$ \textit{full}. Then there is a function $\phi_{\mathcal S}$ from $[r]$ to $\mathcal S$ such that for all $y\in [r]$, $y\in \phi_{\mathcal S}(y)$. By Multicolor Erd\H{o}s-Hajnal, there exists $\eps_{\mathcal S}>0$ such that for any $n$ and any coloring $\chi : E(K_n)\to \mathcal S$ either $\chi$ contains a copy of $\phi_{\mathcal S}\circ \chi_H$ or there is a subset $U\subseteq [n]$ of size $n^{\eps_{\mathcal S}}$ and a color $T\in \mathcal S$ such that for all $u_1,u_2\in U$, $\chi (u_1,u_2)\neq T$.

%Let $S$ be any subset of $2^{[r]}$ \lps{where for each $y\in [r]$, there is a set in $S$ which contains $y$}{such that $\bigcup S=[r]$}. Then there is a function $\phi_S$ from $[r]$ to $S$ \lps{where each $y$ is sent to an element of $S$ which contains $y$}{such that for all $y\in [r]$, $y\in \phi_S(y)$}. Let $\eps_S$ be the value such that in a graph with $n$ vertices where each edge is colored one of the $|S|$ colors, either in the graph there is a copy of $H^*$ with $\phi_S$ applied to all edge colors, or there is some subset of vertices of size $n^{\eps_S}$ and a color $T$ such that all edges in that subset aren't colored $T$. Such an $\eps_S$ exists by Multicolor Erd\H{o}s-Hajnal.
%\lp{Did you state E-H? Also, I would reverse this: say}\\
%\textcolor{purple}{
%``By E-H, there exists $\eps_{\mathcal S}>0$ such that for any coloring $\chi:E(K_n)\to \mathcal S$ then either $\chi$ contains a copy of $\phi_{\mathcal S}\circ \chi_H$ or there is a subset $U\subseteq [n]$ of size $n^{\eps_{\mathcal S}}$ and color $T\in \mathcal S$ such that for all pairs $u_1,u_2\in U$, $\chi(u_1,u_2)\neq T$."}
%\lp{Also I don't understand $\eps$ below, didn't you say $\eps_S$ is always defined?}

Let $\eps=\min\{\eps_\mathcal S|\mathcal S\text{ full}\}$. Since $M$ is sufficiently large with respect to $H$, we have $M>\eps^{-2^r}$. Let $\mathcal{S}_1$ be the subset of $2^{[r]}$ with the fewest elements such that $\mathcal{S}_1$ is full and there exists a subset $W\subseteq [N]$ of $k^{\eps^{-|\mathcal{S}_1|}}$ vertices of $G^*$ where $\chi_G(W)\subseteq \mathcal{S}_1$. Such a smallest family exists because by the construction of $M$, the family $\mathcal{S}=2^{[r]}$ satisfies these properties with $W=V(G^*)$. Then from the above paragraph, since by assumption there is no copy of $\phi_S\circ \chi_H$ inside $W$, there must be some $T\in \mathcal{S}_1$ and $U\subset W$  of size $k^{\eps_{\mathcal{S}_1}\eps^{-|S_1|})}$ vertices such that all edges in $U$ are not colored $T$. Let $S_2=S_1\setminus \{T\}$. Note that $\eps_{\mathcal{S}_1}\eps^{-|\mathcal{S}_1|}>\eps^{1-|\mathcal{S}_1|}=\eps^{-|\mathcal{S}_2|}$. Since $\mathcal{S}_1$ is minimal, then $\mathcal{S}_2$ is not defined, so there is some $y\in [r]$ such that $y\notin \bigcup S_2$. Thus, among the columns of $G$ corresponding to $U$, there is no edge in row $y$. Since $|U|\ge k^{\eps^{-|\mathcal S_2|}}\ge k$, this gives a contradiction, as desired.
%\lp{This last paragraph was rough, Give more names to sets of things please. I made some edits below.}\\
%\textcolor{purple}{
%Let $\eps=\min\{\eps_S:\mathcal S\subseteq 2^{[r]},\ \bigcup S=[r]\}$ and let $M>\eps^{-2^r}$. Let $\mathcal{S}_1$ be the subset of $2^{[r]}$ with the fewest elements such that $\bigcup\mathcal S_1=[r]$ and there is a subset $V'\subseteq [N]$ of size $k^{(1/\eps)^{|\mathcal{S}_1|}}$ such that $\chi_G(V')\subseteq \mathcal{S}_1$. Such a smallest family exists because by construction of $M$, the family $\mathcal S_1=2^{[r]}$ satisfies these properties with $V'=V$. Then from the above paragraph, since by assumption there isn't a copy of $\phi_{\mathcal S}\circ \chi_H$ inside $V'$, there must be some $T\in \mathcal{S}_1$ and a subset $U\subseteq V'$ of size $k^{\eps_{\mathcal{S}_1}/(\eps^{\mathcal{S}_1})}$ such that all edges in $U$ are not colored $T$. Let $\mathcal{S}_2=\mathcal{S}_1\setminus \{T\}$. Note that $\eps_{\mathcal{S}_1}/\eps^{\mathcal{S}_1}>1/\eps^{S_2}$. Since $\mathcal{S}_1$ is minimal, $\bigcup \mathcal S_2\neq [r]$, so there is some $y\in [r]$ such that $y$ is not in any of the colors of pairs in $U$. Thus, $U$ induces a horizontal coclique in $G$ in row $y$. Since $k^{(1/\eps)^{|S_2|}}\ge k$, this gives a contradiction, as desired.\\
%}
\end{proof}

We are ready to prove our main conditional result, \cref{thm:no horiz alignment}.

\begin{proof}
Consider an $N\times N$ grid subgraph $G$ with no coclique of size $k$. By applying \cref{prop: supersaturation} to a vertical $r$-clique, there exists some $\alpha,\beta,\beta'$ such that if $N>ak^\beta$, there are at least $\alpha^{-1}N^{r+1}k^{-\beta'}$ vertical $r$-cliques in $G$. There are $\binom{N}{r}<N^r$ choices of $r$ rows that these $r$-cliques are in, so some set $S$ of $r$ rows has at least $\alpha^{-1}Nk^{-\beta'}$ vertical $r$-cliques which lie in rows of $S$.

From \cref{MEH lemma}, if $N>k^M$ for some $M$ in terms of $H$, restricting to the rows of $S$ and the $r$-cliques within them suffices.
\end{proof}

As stated in the introduction, this verifies \cref{conj:squarefree} if $H$ lies in at most $2$ rows (or columns).

\begin{corollary}
\label{thm: two row grids}Assume the Multicolor Erd\H{o}s-Hajnal conjecture is true.
If $H$ is a square-free grid subgraph on $2$ rows, then $H\in \pgr$.
\end{corollary}
\begin{proof}
Let $H$ be on $c$ columns, suppose they are labeled by elements of $[c]$. If $H$ has no vertical edges, then \cref{thm:no horiz alignment} suffices. Suppose otherwise.

Suppose $c_1$ columns don't have a vertical edge. By swapping columns, without loss of generality let these be columns $[1,c_1]$. We allow $c_1=0$. From a single vertex, we can repeatedly column-bridge $c_1+1$ times to construct a horizontal $(c_1+1)$-clique, so a horizontal $(c_1+1)$-clique satisfies the conditions of \cref{Cor: col/row Dup GridGraph}.

Let $H_1$ be the graph on $2$ rows and $c_1+1$ columns where every row edge exists and the edge in column $c_1+1$ exists. Then $H_1$ is formed from row-bridging of a horizontal $(c_1+1)$-clique. Thus, it also satisfies the conditions of \cref{Cor: col/row Dup GridGraph}.

We can strengthen column-bridging with \cref{MEH lemma}. In an $N\times N$ grid subgraph with no coclique of size $k$, by \cref{Cor: col/row Dup GridGraph}, there are at least $\frac{N^{c_1+3}}{k^a}$ embeddings of $H_1$ for some constant $a$ if $N>k^b$ for some constant $b$. Hence, there are columns $x_1,x_2,\dots, x_{c_1}$, and rows $y_1,y_2$ such that $\frac{N}{k^a}$ embeddings $(\phi_c,\phi_r)$ of $H_1$ into $G$ send the first $c$ columns to $x_1,x_2,\dots, x_{c_1}$ and the two rows to $y_1$ and $y_2$. These $\frac{N}{k^a}$ embeddings all send column $c_1+1$ of $H_1$ to different columns. Let this set of columns be $S$.

Let $H_2$ be the grid subgraph which is $H$ restricted to columns in $(c_1,c]$ with all column edges deleted. No edges in $H_2$ share the same two columns. Let $N>k^M$ for some large $M$ dependent on $H_2$. Then $|S|>k^{M-a}$. By \cref{MEH lemma}, in the two rows $y_1,y_2$ and among the columns in $S$, there is a copy of $H_2$. The induced grid subgraph made by the vertices of this copy as well as the vertices $(x_a,y_b)$ for $a\in [c_1],b\in [2]$ contains a copy of $H$, as desired.
\end{proof}

\section{Implications for Hypergraph Ramsey numbers}\label{sec:hypergraph}

As observed in \cite{conlon2023hypergraph}, upper bounds on grid Ramsey numbers imply corresponding upper bounds for certain off-diagonal hypergraph Ramsey numbers.

Before proceeding with this section, we will recall the standard definition of embeddings in $3$-graphs and define relevant notation.
\begin{definition}[$3$-graph embedding]
  An embedding of a $3$-graph $H_* = ([m],E_{H_*})$ into a $3$-graph $G_* = ([n],E_{G_*})$ is an injective graph homomorphism $\varphi: [m] \to [n]$. That is, for all vertices $i,j,k \in [m]$, $\{\varphi(i),\varphi(j),\varphi(k)\} \in E_{G_*}$ if $\{ i,j,k \} \in E_{H_*}.$

  Denote by $t_3(H_*,G_*)$ the number of embeddings of $H_*$ in $G_*$. Note that $t_3(H_*,G_*)=t_g(f_g(H_*),f_g(G_*))$.
\end{definition}

% \cref{cor:tight cycle theorem} will be proved in this section.
We begin by proving \cref{thm: vertexDup3-Graphs}, a direct analogue of \cref{Cor: col/row Dup GridGraph}, in which we see that row/column-bridging in the grid setting corresponds to an operation we call \emph{vertex-bridging}, which may be thought of as a vertex blowup boosted by the addition of an extra edge.

\begin{prop}[Vertex-bridging on general hypergraphs]\label{thm: vertexDup3-Graphs}
  Let $G = (V,E)$ be a $3$-graph on $n$ vertices such that its complement $3$-graph $\overline{G}$ does not contain $S_k^{(3)}$. Also, let $H = (V_H,E_H)$ be a $3$-graph on $m$ vertices. Finally, let $H' = (V_{H'},E_{H'})$ be a $3$-graph on $m+1$ vertices formed by taking $H$, blowing up any $v \in V_H$ to vertices $v,v'$, and taking $E_{H'} = E_H \sqcup \{ v,v',w \}$ for any chosen $w \in V_H \setminus \{v\}$. Suppose that there exists constants $a,c',c > 0$ such that if $n > ak^c$, then $t_3(H,G) > a^{-1} \cdot n^mk^{-c'}$. Then, if $n > 2ak^{c'+1}$, then $t_3(H',G) > (2a)^{-2}n^{m+1}k^{-(2c'+1)}$.
\end{prop}
\begin{proof}

The proof is analogous to that of \cref{thm: Constructible by Duplication is Polynomial}.

Assume $n >2ak^{c'+1}$. For any injection $\varphi:V_H\setminus \{v\}\to V$, define $P_\varphi$ as the set of $u \in V\setminus \rm{Im}(\varphi)$ such that the injection $\varphi\cup\{(v,u)\}$ is an embedding of $H$. Thus, by assumption,
\[\sum_{\varphi} |P_{\varphi}| =t_3(H,G)> \frac{n^m}{ak^{c'}}.\]
Now, note that for any given $\varphi$ and a pair $u,u' \in P_\varphi$, an embedding of $H'$ is formed if we have the edge $\{u,u',\varphi(w)\}\in E$. Let $F_{\varphi,w}$ be the link $\varphi(w)$ induced on $P_\varphi$, so for $u,u'\in P_\varphi$, we have $\{u,u'\}\in F_{\varphi,w}$ if and only if $\{ u,u',\varphi(w)\} \in E$. Then, each edge in $F_{\varphi,w}$ forms an embedding of $H'$. Note that $F_{\varphi,w}$ must not contain an independent set $U$ of size $k$, or else $\overline{G}[U\cup\{w\}]$ contains $S_k^{(3)}$. Thus, for any given $\varphi$, Tur\'an's theorem guarantees that there exists at least $f(|P_\varphi|)$ embeddings of $H'$ in $G$ using all vertices of $S$, where $f(x) = \frac{x(x-k+1)}{2(k-1)}$. Note that $f(x)$ is convex and $f(x) \geq x^2/(4k)$ for $x \geq 2k$. Since there are $c_{n,m}:=\binom{n}{m-1}(m-1)! \leq n^{m-1}$ choices of $\varphi$, the average cardinality of $P_\varphi$ is greater than $a^{-1} \cdot n k^{-c'}>2k$. Therefore, by Jensen's inequality,
\[t_3(H',G) \geq \sum_S f(|P_S|) \ge
c_{n,m}f\Big(\frac{1}{c_{n,m}}\sum_{\varphi}|P_\varphi|\Big)\geq
\frac{c_{n,m}}{4k}\Big(\frac{n^m}{c_{n,m}ak^{c'}}\Big)^2
\geq n^{m+1}(2a)^{-2}k^{-2c'-1},
%\binom{n}{m-1}(m-1)! \cdot \frac{\left( \frac{n}{a \cdot k^{c'}} \right) \left( \frac{n}{a \cdot k^{c'}} - k \right)}{2(k-1)} \geq \binom{n}{m-1}(m-1)! \cdot \frac{\left( \frac{n}{a \cdot k^{c'}} \right)^2}{4k}
\]
achieving our desired result.
\end{proof}

A similar result to \cref{thm: Constructible by Duplication is Polynomial} can be proved in the Hypergraph setting via iteration of vertex-bridging.

\begin{prop}\label{thm: Hypergraph Constructible by Duplication is Polynomial}
If $H$ can be obtained by finite iteration of vertex-bridging starting from the graph on one vertex, $r(H,S_k^{(3)}) = k^{O(2^m)}$.
\end{prop}
\begin{proof}
The proof is identical to the proof in \cref{thm: Constructible by Duplication is Polynomial}, with $r+c$ replaced by $m$.
\end{proof}

We also obtain a cubic bound on $r(C_6^{(3)}, S_k^{(3)})$.
%This follows from \cref{Thm: Bounds on Ell} when we use $f_g^{-1}$ to convert into the grid setting. \ra{(Write as complete sentence}\krishna{ add reference to $f_g^{-1}$ once I or someone else rewrites preliminaries to make it a definition.}
This follows from \cref{Thm: Bounds on Ell} when we
convert into the grid setting.

Now we will present the proof of \cref{cor:tight cycle theorem}. 

\begin{proof}[Proof of \cref{cor:tight cycle theorem}]
We first prove for the case where $t$ is even. Let $t=2d$.

Note that $C^{(3)}_{t}$ is a 3-graph with property B. Let the vertices of the tight cycle be $x_1,y_2,\ldots, x_d,y_d$ in order. Now we apply $f_g$. Note that $f_g (x_{i},y_{i},x_{i+1}$) = $\{(x_i,y_i),(x_{i+1},y_i)\}$ and $f_g (y_{i},x_{i+1},y_{i+1}$)  = $\{(x_{i+1},y_i),(x_{i+1},y_{i+1})\}$. This is an alternating cycle, which we denoted as $\AC_t$. From \cref{Thm:Alternating cycles are polynomial} we get the desired result since $R(C_t^{(3)},S_k^{(3)}) \leq 2\gr(\AC_t , K_k) = k^{O(1)} $ when $ t>4 $.

Now we proceed to the case where t is odd, $t=2d+1$. 

Here we use $f_g$ as in the previous case and end up with a grid graph on columns $x_1,...,x_{2d+1}$ and rows $y_1, ..., y_d$. Now our edges would be :

$\{(x_i,y_i),(x_{i+1},y_i)\}$ for $i \in [d]$

$\{(x_{i+1},y_i),(x_{i+1},y_{i+1})\}$ for $i \in [d-1]$

$\{(x_1,y_1),(x_{d+1},y_1)\}$ and $\{(x_1,y_d),(x_{d+1},y_d)\}$
We denote this graph as $AS_{d-1}$ ($AS$ means Aligned Staircase and $AS_4$ can be found in \cref{fig:corresponding to odd tight cycle}). Note that $R(C_t^{(3)},S_k^{(3)}) \leq 2gr(AS_{d-1} , K_k) $ when $ t\geq 5 $. Therefore, it is sufficient to prove that $gr(AS_{d-1} , K_k) = O(k^{O(1)})$ for $d\geq2$. 
From \cref{thm: Constructible by Duplication is Polynomial}, it suffices to show that we can obtain $AS_{d-1}$ via row/column-bridging starting from a row-clique of order $(d+1)$. This is because a row clique is a generalized subdivision of an edge, and hence bridging-constructible from a vertex.

Let's label our starting row clique as $(x_1,y_1),(x_2,y_1),(x_3,y_1), \dots (x_{d+1},y_1) $. We would label this graph $\chi_{0}$ On the first move we bridge row $y_1$ and make $y_2$ and add the edge $\{(x_2,y_1),(x_2,y_2)\}$. We call this graph $\chi_{1}$.  On the $i^{th} $ move ,for $i \in [d-1]$, we bridge row $y_i$ of $\chi_{i-1}$ and make the new row $y_{i+1}$ and add the edge $\{(x_{i+1},y_i),(x_{i+1},y_{i+1})\}$ which gives us the graph $\chi_{i}$. 

The graph $\chi_{d-1}$ we create after these $d-1$ moves contain $AS_{d-1}$ as a subgraph. Therefore $gr(AS_{d-1},K_k) \leq gr(\chi_{d-1},K_k)=k^{O(1)}$.  

\end{proof}

\remark{In general we have for all $H\subseteq K_{n,n}^{(3)}$, $R(H,S_k^{(3)})\leq 2\gr(f_g(H),k)$ , since $f_g^{-1}(K_k) = S_k^{(3)}$ and $f_g(H) \subseteq G_{n \times n}$, while the number of vertices in $H$ is atmost $2n$.  }

\begin{figure}[H]
 \centering
\begin{tikzpicture}[scale=0.9, mystyle/.style={circle, draw, fill=black, inner sep=3pt}]
 % Nodes
 \node[mystyle] (A) at (0,4) {};
 \node[mystyle] (B) at (1,4) {};
 \node[mystyle] (C) at (1,3) {};
 \node[mystyle] (D) at (2,3) {};
 \node[mystyle] (E) at (2,2) {};
 \node[mystyle] (F) at (3,2) {};
 \node[mystyle] (G) at (3,1) {};
 \node[mystyle] (H) at (4,1) {};
 \node[mystyle] (I) at (4,0) {};
 \node[mystyle] (J) at (5,0) {};
 \node[mystyle] (K) at (0,0) {};
 \node[mystyle] (L) at (5,4) {};

 % Edges – modify or add as needed
 \draw[black,thick] (A) -- (B);
 \draw[black,thick] (C) -- (B);
 \draw[black,thick] (C) -- (D);
 \draw[black,thick] (E) -- (D);
 \draw[black,thick] (E) -- (F);
 \draw[black,thick] (G) -- (F);
 \draw[black,thick] (G) -- (H);
 \draw[black,thick] (I) -- (H);
 \draw[black,thick] (I) -- (J);
 \draw[black,thick] (0,0) .. controls (2.5,-1) .. (5,0);
 \draw[black,thick] (0,4) .. controls (2.5,5) .. (5,4);
 \draw[black,thick,dotted] (L) -- (J);
 \draw[black,thick,dotted] (A) -- (K);

\end{tikzpicture}
\caption{$AS_4$}
\label{fig:corresponding to odd tight cycle}
\end{figure}

\section{Further Directions}\label{sec: further directions}

We conclude with three open questions regarding grid Ramsey numbers.

\begin{question}[N-Z stool]
Consider the grid subgraph $H$ on $4$ columns and $2$ rows drawn below, which we call the N-Z stool (imagine drawing the graph in 3D with each row in a 2D plane).
%The row edges in row 1 are $\{\{1,2\},\{2,3\},\{3,4\}\}$, and the row edges in row $2$ are $\{\{1,3\},\{1,4\},\{2,4\}\}$. All column edges are present.

Is $\gr(H,K_k)$ polynomial in $k$?
\begin{figure}[h]
  \centering
  \begin{tikzpicture}[scale=0.6, mystyle/.style={circle, draw, fill=black, inner sep=3pt}]
  % Nodes
  \node[mystyle] (A) at (1,1) {};
  \node[mystyle] (B) at (2,1) {};
  \node[mystyle] (C) at (3,1) {};
  \node[mystyle] (D) at (4,1) {};
  \node[mystyle] (E) at (1,2) {};
  \node[mystyle] (F) at (2,2) {};
  \node[mystyle] (G) at (3,2) {};
  \node[mystyle] (H) at (4,2) {};

  % Edges – modify or add as needed
  \draw[red,thick] (A) -- (B);
  \draw[red,thick] (B) -- (C);
  \draw[red,thick] (C) -- (D);
  \draw[red,thick] (G) to[out=140,in=40] (E);
  \draw[red,thick] (E) to[out=50, in=130] (H);
  \draw[red,thick] (H) to[out=140, in=40] (F);
  \draw[red, thick] (A) -- (E);
  \draw[red, thick] (B) -- (F);
  \draw[red, thick] (C) -- (G);
  \draw[red, thick] (D) -- (H);

  \end{tikzpicture}
\end{figure}
\end{question}
By \cref{thm:no horiz alignment}, this is true conditional on Multicolor Erd\H{o}s-Hajnal. However, we can't obtain this grid subgraph through row/column-bridging starting from a single edge: consider a grid subgraph $H_1$ which does not have $H$ as a sub-pattern but does have $H$ as a sub-pattern after a bridging step. Then it could not have bridged a row, since there are $4$ edges between the two rows of $H$. Hence, it bridged a column, and it can be checked that there must have been a square in $H_1$.

%\lp{add peterson stool too? Also I like the name NZ-stool, it's fun.}

\begin{question}
Consider the grid subgraph $H$ drawn below.
%,which is a grid pattern on $3$ rows and $3$ columns and is the path $(1,2)\sim (1,3)\sim (2,3)\sim (2,2)\sim (3,2)\sim (3,1)\sim (2,1)$.
Is $\gr(H,K_k)$ polynomial?
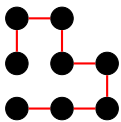
\begin{figure}[h]
  \centering
  \begin{tikzpicture}[scale=0.6, mystyle/.style={circle, draw, fill=black, inner sep=3pt}]
  % Nodes
  \node[mystyle] (A) at (1,1) {};
  \node[mystyle] (B) at (2,1) {};
  \node[mystyle] (C) at (3,1) {};
  \node[mystyle] (D) at (1,2) {};
  \node[mystyle] (E) at (2,2) {};
  \node[mystyle] (F) at (3,2) {};
  \node[mystyle] (G) at (1,3) {};
  \node[mystyle] (H) at (2,3) {};

  % Edges – modify or add as needed
  \draw[red,thick] (D) -- (G) -- (H) -- (E) -- (F) -- (C) -- (B) -- (A);

  \end{tikzpicture}
\end{figure}
\end{question}
Neither row/column-bridging nor our results conditional on the Multicolor Erd\H{o}s-Hajnal conjecture work for this graph. Changing $H$ by deleting any edge makes $\gr(H,K_k)$ polynomial through bridging.

\begin{question}
  Is it true that $\gr(H, K_k)$ is polynomial for any $H$ which does not contain a copy of $G_{2 \times 2}$?
\end{question}
\cref{thm: two row grids} gives some weak evidence towards the answer being yes, but we do not have a strong belief either way.

%\lp{A remark: If $H$ does not contain a square and $H'$ is obtained by a row/column duplication of $H$ (with added edge), then $H'$ does not contain a square. Thus, the class $\mathcal H$ of square free patterns is closed under duplication, which is a necessary condition for 8.3 to hold. In other words, Question 8.3 cannot be \textit{disproved} by constructing a square out of square-free graphs.}

\vspace{3mm}

{\noindent \textbf{Acknowledgments.}} The authors are grateful to Dhruv Mubayi for suggesting this problem, to Jacob Fox for stimulating conversations about the multicolor Erd\H{o}s-Hajnal conjecture, and to Ruben Ascoli and Logan Post for many helpful comments on this manuscript. We are also grateful to the Georgia Tech Research Experiences for Undergraduates program, where most of this research was conducted.

%\bibliographystyle{unsrtnat}
%\bibliography{biblio}

\bibliographystyle{abbrvnat}   % or abbrvnat, unsrtnat, amsplain, plainnat
\bibliography{biblio}          % your .bib file (no .bib extension)

\vspace{0.5cm}

\begin{comment}

\noindent[1] M. Ajtai, J. Komlós and E. Szemerédi, A note on Ramsey numbers, \textit{J. Combin. Theory Ser. A} \textbf{29} (1980), 354-360.

\vspace{\bibsep}

\noindent[2] D. Conlon, J. Fox, X. He, D. Mubayi, A. Suk and J. Verstraëte, Hypergraph Ramsey numbers of cliques versus stars, \textit{Random Structures \& Algorithms} \textbf{63} (2023), 610-623.

\vspace{\bibsep}

\noindent[3] P. Erdős and A. Hajnal, Ramsey-type theorems, \textit{Discrete Appl. Math.} \textbf{25} (1989), 37-52.

\vspace{\bibsep}

\noindent[4] P. Erdős and M. Simonovits, Supersaturated graphs and hypergraphs, \textit{Combinatorica} \textbf{3} (1983), 181-192.

\vspace{\bibsep}

\noindent[5] P. Erdős, A. Hajnal and R. Rado, Partition relations for cardinal numbers, \textit{Acta Math. Acad. Sci. Hungar.} \textbf{16} (1965), 93-196.

\vspace{\bibsep}

\noindent[6] J. Fox and X. He, Independent sets in hypergraphs with a forbidden link, \textit{Proc. London Math. Soc.} \textbf{123} (2021), 384-409.

\vspace{\bibsep}

\noindent[7] J. Fox, A. Grinshpun and J. Pach, The Erdős–Hajnal conjecture for rainbow triangles, \textit{J. Combin. Theory Ser. B} \textbf{111} (2015), 75-125.

\vspace{\bibsep}

\noindent[8] J. H. Kim, The Ramsey number $R(3, t)$ has order of magnitude $t^2/\log t$, \textit{Random Structures \& Algorithms} 7 (1995), 173-207.

\end{comment}

\end{document}